\documentclass[12pt]{amsart}
\usepackage{amsmath,amssymb,amsfonts,amsthm,amscd,indentfirst}
\usepackage{amsmath,amssymb,amsfonts,amsthm,amscd}
\usepackage{indentfirst}
\usepackage{hyperref}
\usepackage{enumitem}

\numberwithin{equation}{section}

\def\dsum{\displaystyle\sum}
\def\dfrac{\displaystyle\frac}

\newtheorem{prop}{Proposition}[section]
\newtheorem{theo}[prop]{Theorem}
\newtheorem{lemm}[prop]{Lemma}

\def\and{\quad{\rm and}\quad}

\def\<{\langle}
\def\>{\rangle}

\usepackage{graphicx}

\begin{document}

\title[Hypersurfaces of constant sum Hessian curvature]{Hypersurfaces of constant sum Hessian curvature in Hyperbolic space}
\author[Jianbo Yang]{Jianbo Yang}
\address{College of Science, Harbin University of Science and Technology, 52 Xuefu Road, Harbin, 150080, Heilongjiang Province, China.}
\email{2320800008@stu.hrbust.edu.cn}
\author[Yueming Lu]{Yueming Lu}
\address{College of Science, Harbin University of Science and Technology, 52 Xuefu Road, Harbin, 150080, Heilongjiang Province, China.}
\email{yueminglu@hrbust.edu.cn}
\begin{abstract}
In this paper, we study the asymptotic Plateau problem  in hyperbolic space for constant sum Hessian curvature. More precisely, given a asymptotic boundary $\Gamma$, one seeks a complete hypersurface $\Sigma$ in $\mathbb{H}^{n+1}$ satisfying $\sigma_{n-1}(\kappa)+\alpha\sigma_{n}(\kappa)=\sigma\in (0,n),\,\,\partial \Sigma=\Gamma$ where $\alpha$ is a non-negative number.
\end{abstract}

\maketitle 
\tableofcontents

\section{Introduction}
\subsection{Background}
Let $\Sigma$ be a smooth surface in $\mathbb{R}^{3}$. Consider a deformation $\Sigma_{t}$ in the interior and the area satisfies \begin{equation*}
    A(\Sigma_{t})\geq A(\Sigma).
\end{equation*}If for all such deformations we have\begin{equation*}
    \frac{d}{dt}A(\Sigma_{t}) \bigg |_{t=0}=0,
\end{equation*} $\Sigma$ be referred to as \emph{minimal surface}. A fundamental problem concerns the existence of such minimal surfaces. More generally, the existence of minimal submanifolds $\Sigma$ in arbitrary Riemannian manifolds $M$ constitutes a fundamental problem in Riemannian geometry. The cases $M=\mathbb{R}^{n}$ and $M=\mathbb{S}^{n}$ have been thoroughly researched by many authors and require no further discussion here.
\par It is noteworthy that the case $M=\mathbb{H}^{n+1}$, where we use the upper half-space model \begin{equation*}
    \mathbb{H}^{n+1}=\{(x,x_{n+1})\in \mathbb{R}^{n+1}:x_{n+1}>0\}
\end{equation*} equipped with the hyperbolic metric \begin{equation*}
    ds^{2}=\frac{1}{x_{n+1}^{2}}\sum_{i=1}^{n+1}dx_{i}^{2}.
\end{equation*}More precisely, let $\partial_{\infty}\mathbb{H}^{n+1}$ be the ideal boundary of $\mathbb{H}^{n+1}$ at infinity. Given a closed embedded smooth $(n-1)-$dimensional  submanifold $\Gamma\subset \partial_{\infty}\mathbb{H}^{n+1}$, one seeks a complete hypersurface $\Sigma$ in $\mathbb{H}^{n+1}$ satisfying \begin{equation}\label{problem(1.1)}
    f(\kappa)=\sigma,\,\,\,\,\partial\Sigma=\Gamma,
\end{equation}where $f$ is a smooth symmetric function of $n$ variables, $\kappa=(\kappa_{1},\cdots,\kappa_{n})$ are the principal curvatures of $\Sigma$ and $\sigma\in(0,1)$. Meanwhile, we may also consider the Poincaré ball model of hyperbolic space. Indeed, the two models are equivalent; for computational convenience we prefer the upper half-space model.
\par In 1980s, the asymptotic Plateau problem \eqref{problem(1.1)} for area-minimizing varieties was first addressed by Anderson \cite{A82,A83} employing geometric measure theory. Following this work, Tonegawa \cite{T} achieved an extension to constant mean curvature hypersurfaces. PDE theory provides a powerful framework for addressing such problems. Lin \cite{Lin,Lin1} pioneered a purely PDE-theoretic proof of Anderson's existence-uniqueness theorem while also establishing regularity properties of the solutions. Later, the asymptotic Plateau problem for constant mean curvature was investigated via PDE methods in the works of Nelli and Spruck \cite{NS} and Guan–Spruck \cite{GS00}. Labourie \cite{L} solved the case of hypersurfaces with constant Gauss curvature in $\mathbb{H}^3$. More generally, Rosenberg and Spruck \cite{RS} extended this result to hypersurfaces in $\mathbb{H}^{n+1}$. Guan, Spruck, Szapiel, and Xiao \cite{GSS, GS11, GSX} resolved the asymptotic Plateau problem for a broad class of functions $f$ defined on the positive cone $\Gamma_n$. For general $f$ satisfying some structure conditions defined on an open symmetric convex cone $K$ such that $\Gamma_{n}\subset K$, the asymptotic Plateau problem was solved by Guan and Spruck \cite{GS10} if $\sigma>\sigma_0$, where $0.3703<\sigma_0<0.3704$.
\par Building on the seminal work of Guan and Spruck \cite{GS10}, we consider specific types of $f$ defined on an certain open symmetric convex cone $K$ such that our results hold for all $\sigma\in(0,1)$. Lu \cite{S. Lu} studied the case where $f=\sigma_{n-1},\,\,K=\Gamma_{n-1}$ by using Ren-Wang's concavity inequality \cite{RW}.  In a recent preprint \cite{wang3}, Wang  was able to obtain the curvature estimates  for problem \eqref{problem(1.1)} with $f=\sigma_{n-2},\,\,K=\Gamma_{n-2}$ by using Ren-Wang's important concavity inequality \cite{RW1}. In \cite{wang3}, Wang also studied the case where $f=\left(\frac{H_{k}}{H_{l}}\right)^{\frac{1}{k-l}},\,1\leq k<l\leq n,\,\,K=\Gamma_{n}$ with curvature bounded from blew. Here $H_{k}$ is the normalized $k$-th elementary symmetric polynomial. The last remaining interesting case is $f=\sigma_{2},\,\,K=\Gamma_{2}$. Wang in \cite{Wang2} was able to solve this case in all dimensions by using almost-Jacobi inequality due to Shankar-Yuan \cite{RY}. For the case $f=\sigma_{k}$ with $2<k<n-2$, Hong-Wang was able to solve the problem \eqref{problem(1.1)} for semi-convex hypersurface in $\mathbb{H}^{n+1}$. However, the case for $f=\sigma_{k},\,\,K=\Gamma_{k}$ with $2<k<n-2$ remains open.
\par Very recently, Chen-Sui-Sun \cite{DZL} investigated another class of special curvature functions $f$, i.e. 
\begin{equation*}
    f(\kappa)=\textstyle \prod_{1 \leq i_{1}< i_{2}<\cdots <i_{p} \leq n}(\kappa_{i_{1}}+\kappa_{i_{2}}+\cdots+\kappa_{i_{p}})
\end{equation*}defined on $p$-convex cone \begin{equation*}
    K=\mathcal{P}_{p}=\{\kappa\in \mathbb{R}^n:\forall 1\leq i_{1} <i_{2}<\cdots <i_{p}\leq n,\kappa_{i_{1}}+\kappa_{i_{2}}+\cdots+\kappa_{i_{p}}>0\}.
\end{equation*}In \cite{DZL}, Chen-Sui-Sun was able to solve the problem \eqref{problem(1.1)} for the case $f=\mathcal{F}_{2},\,\,K=\mathcal{P}_{2}$ in $\mathbb{H}^{4}$. For general dimensions, the authors believe that the key to problem \eqref{problem(1.1)} lies in further exploring the structure of the curvature function $f$.
\subsection{Main results} In this paper, we consider the problem \eqref{problem(1.1)} for the case $f=S_{n},\;n>1$, where \begin{equation*}
    S_{n}=\sigma_{n-1}+\alpha\sigma_{n}.
\end{equation*}As stated in  Guan and Spruck \cite{GS10}, the key to problem \eqref{problem(1.1)} for the case $f=S_{n}$ lies in establishing curvature estimates for all $\sigma \in (0,n)$. We now provide some necessary notation.
\par Denote by \begin{equation*}
    S_{n}^{ii}(\lambda)=\frac{\partial S_{n}(\lambda)}{\partial \lambda_{i}},\;\;\;S_{n}^{ii,jj}(\lambda)=\frac{\partial^{2}S_{n}(\lambda)}{\partial\lambda_{i}\partial\lambda_{j}}.
\end{equation*}We denote Garding's cone by 
\begin{equation*}
    \Gamma_{k}=\{\lambda\in\mathbb{R}^{n}|\sigma_{1}(\lambda)>0,\cdots,\sigma_{k}(\lambda)>0\}.
\end{equation*}As point out in Li-Wang-Ren\cite{LRW}, the admissible  set of the $S_{n}$ is 
\begin{equation*}
    \widetilde{\Gamma}_{k}=\Gamma_{k-1}\cap \{\lambda\in\mathbb{R}^{n}|S_{k}(\lambda)>0\}.
\end{equation*}This means that $\widetilde{\Gamma}_{k}$ is a convex set. In the cone $\widetilde{\Gamma}_{k}$, $S_{n}^{ii}$ is elliptic.

The main results in the paper are as follows: 
\begin{theo}\label{theo1}
For $n\geq 2$, let $\Gamma=\partial\Omega$, where $\Omega$ is a bounded smooth domain in $\mathbb{R}^n$ with nonnegative mean curvature and $\sigma\in (0,n)$. Suppose $\Sigma$ is a $C^4$ vertical graph over $\Omega$ in $\mathbb{H}^{n+1}$ satisfying
\begin{align}\label{S_n}
S_{n}(\kappa)=\sigma,\quad \partial\Sigma=\Gamma,
\end{align}
where $\kappa\in \widetilde{\Gamma}_{n}$. If there exists a constant $A>0$ such that $\sigma_{n}(\kappa)>-A$, then we have
\begin{align*}
\max_{x\in \Sigma, 1\leq i\leq n}|\kappa_i(x)|\leq C+C\max_{x\in \partial\Sigma, 1\leq i\leq n}|\kappa_i(x)|,
\end{align*}
where $C$ is a constant depending only on $n, \Omega$, $A$, $\alpha$ and $\sigma$.
\end{theo}Note that when $\alpha=0$, problem \ref{problem(1.1)} for the case $f=\sigma_{n-1},\;K=\Gamma_{n-1}$ has been settled by Lu\cite{S. Lu}. Thus in this paper we assume $\alpha$ is a positive number. Theorem \ref{theo1} provides a curvature estimate for problem \ref{problem(1.1)} under the condition that the $\sigma_{n}(\kappa)$ is bounded below. As pointed out in  Guan and Spruck \cite{GS10}, we have the following conclusion:
\begin{theo}\label{theo2}
For $n\geq 2$, let $\Gamma=\partial\Omega$, where $\Omega$ is a bounded smooth domain in $\mathbb{R}^n$ with nonnegative mean curvature and $\sigma\in (0,n)$. If there exists a constant $A>0$ such that $\sigma_{n}(\kappa)>-A$, then there exists a complete hypersurface $\Sigma$ in $\mathbb{H}^{n+1}$ satisfying
\begin{align*}
S_{n}(\kappa)=\sigma,\quad \partial\Sigma=\Gamma.
\end{align*}
\end{theo}
\subsection{Outline of proof}
Regarding the proof of Theorem \ref{theo1}, We select a new test function proposed by Lu \cite{S. Lu} and apply the maximum principle. To bound $\kappa_{1}$, we exploit the maximum point of the test function, reducing its second-order differential inequalities through cancellation of: 
    \begin{enumerate}
    \item fourth-order terms,
    \item negative third-order terms, 
    \item some other negative second-order terms.
\end{enumerate}
\par For fourth-order terms, we can be handled directly via the equation itself without additional techniques. \par The more challenging part is the handling of the negative third-order terms, which is highly dependent on the structure of the sum Hessian operator $S_{n}$. In this paper, we employ the concavity inequality established by Li-Ren \cite{PR} to effectively address the bad third-order terms.
\par Due to the special structure of $S_{n}$, very bad second-order terms $-\kappa_{1}\sigma_{n-1}$ arise in our proof. To overcome this difficulty, we impose the condition that $\sigma_{n}$ has a lower bound. However, the problem \ref{problem(1.1)} for the case $f=S_{n},\;K=\widetilde{\Gamma}_{n}$ without imposing any additional conditions remains open. Furthermore, we can also consider the problem \ref{problem(1.1)} for the case $f=S_{k},\;K=\widetilde{\Gamma}_{k}$ with $2\leq k \leq n-1$. \par As pointed out in Guan and Spruck \cite{GS10}, based on the curvature estimate obtained in Theorem \ref{theo1}, we can directly derive Theorem \ref{theo2}.
\section{Preliminaries}
\par In this section, we compile the fundamental formulas for hypersurfaces in hyperbolic space along with key properties of the operator $S_{n}$, see \cite{GS10,PR}.
\subsection{Properties of sum Hessian $S_{n}$}
\begin{lemm}\label{lem2.1}
		Let $\lambda=(\lambda_1,\cdots,\lambda_n)\in
		\mathbb{R}^n$, we have\\
		\par
		(i) $S_n^{ii}(\lambda):=\dfrac{\partial
			S_n(\lambda)}{\partial\lambda_i}=S_{n-1}(\lambda|i)$,
		\quad $i=1,2,\cdots,n$;
		\par
		(ii)
		$S_n^{ii,jj}(\lambda):=\dfrac{\partial^2S_n(\lambda)}{\partial\lambda_i\partial\lambda_j}=S_{n-2}(\lambda|ij)$,
		\quad $i,j=1,2,\cdots,n$, \quad and $S_k^{ii,ii}(\lambda)=0$;
		\par
		(iii) $S_n(\lambda)=\lambda_iS_{n-1}(\lambda|i)+S_n(\lambda|i),$
		$i=1,\cdots,n$;
		\par
		(iv)
		$\dsum_{i=1}^nS_n(\lambda|i)=\sigma_{n-1}(\lambda);$
		\par
		(v)
		$\dsum_{i=1}^n\lambda_iS_{n-1}(\lambda|i)=nS_n(\lambda)-\sigma_{n-1}(\lambda).$
	\end{lemm}
\begin{lemm}\label{lem2.2}	For $\lambda=(\lambda_1,\cdots,\lambda_n)\in\widetilde{\Gamma}_n$, $\lambda_1\ge\cdots\ge\lambda_n.$
		If $i=1,2,\cdots,n-1$, then exist a positive constant $\theta$ depending on $n$,  such that
		$$S_n^{ii}\ge\frac{\theta S_n(\lambda)}{\lambda_i}.$$
	\end{lemm}
The following concavity inequality by Li-Ren\cite{PR} plays a crucial role in the proof of our main theorem.
\begin{lemm}\label{lem2.3}
	If $\lambda\in\widetilde{
		\Gamma}_n$, then $\forall \epsilon>0$, $\exists K(\epsilon)$, such that
	\begin{equation}\label{y1}
		\lambda_1\left(K\left(\sum\limits_jS_n^{jj}(\lambda)\xi_j\right)^2-S_n^{pp,qq}(\lambda)\xi_p\xi_q\right)-S_n^{11}(\lambda)\xi_1^2+(1+\epsilon)\sum\limits_{j> 1}S_n^{jj}\xi_j^2\ge0.
	\end{equation}
\end{lemm}

\subsection{Hypersurfaces in the hyperbolic space $\mathbb{H}^{n+1}$}
For computational convenience, we adopt the upper half-space model of hyperbolic space in this paper. More precisely, \begin{equation*}
    \mathbb{H}^{n+1}=\{(x,x_{n+1})\in \mathbb{R}^{n+1}:x_{n+1}>0\}
\end{equation*} equipped with the hyperbolic metric \begin{equation*}
    ds^{2}=\frac{1}{x_{n+1}^{2}}\sum_{i=1}^{n+1}dx_{i}^{2}.
\end{equation*}The boundary at infinity of $(n+1)$-dimensional hyperbolic space, $\partial_\infty\mathbb{H}^{n+1}$, is canonically identified with $\mathbb{R}^n$, where $\mathbb{R}^n$ is embedded in $\mathbb{R}^{n+1}$ as the subspace $\mathbb{R}^n \times {0}$. We consider a connected, orientable, complete hypersurface $\Sigma$ embedded in $\mathbb{H}^{n+1}$ whose asymptotic boundary at infinity is compact. Denote by $\mathbf{n}$ the unit normal vector field along $\Sigma$ oriented toward the non-compact component of $\mathbb{R}^{n+1}_+\setminus \Sigma$. Let $g$ be the hyperbolic metric on $\Sigma$ and $\nabla$ be the Levi-Civita connection with respect to $g$. On the other hand, we may regard $\Sigma$ as a hypersurface in $\mathbb{R}^{n+1}$. Let $\tilde{g}$ be the induced metric on $\Sigma$ from $\mathbb{R}^{n+1}$ and $\tilde{\nabla}$ be the Levi-Civita connection with respect to $\tilde{g}$. $\overline{\nabla}$ denotes the Levi-Civita connection of $\mathbb{H}^{n+1}$. \par Now we regard $\Sigma$ as a hypersurface in $\mathbb{R}^{n+1}$. We suppose $\Sigma$ is the vertical graph of a function $x_{n+1}=u(x_{1},\cdots,x_{n})$. Let $X$ and $\nu$ denote the position vector and the outer unit normal vector of $\Sigma$ in $\mathbb{R}^{n+1}$, respectively. We define \begin{equation*}
    u=X\cdot \mathbf{e}, \quad \nu^{n+1}=\nu\cdot\mathbf{e},
\end{equation*}where $\mathbf{e}$ is the unit vector field oriented in the positive $x_{n+1}$ direction within $\mathbb{R}^{n+1}$, with $\cdot$ designating the Euclidean inner product in $\mathbb{R}^{n+1}$. Viewing $\Sigma$ as a submanifold of $\mathbb{H}^{n+1}$. The first fundamental form  and second fundamental form on $\Sigma$ are then given by \begin{equation*}
    \begin{split}
        g_{ij}=\langle X_{i},X_{j} \rangle&=\langle \frac{\partial}{\partial x_{i}}+u_{i}\frac{\partial}{\partial x_{n+1}} ,\frac{\partial}{\partial x_{j}}+u_{j}\frac{\partial}{\partial x_{n+1}}\rangle\\
        &=\frac{1}{u^{2}}\left(\delta_{ij}+u_{i}u_{j}\right)=\frac{\tilde{g}_{ij}}{u^{2}}.
    \end{split}
\end{equation*}\begin{equation*}
    h_{ij}=\langle \overline{\nabla} _{X_{i}}X_{j}, \mathbf{n}\rangle=\frac{\tilde{h}_{ij}}{u}+\frac{\nu^{n+1}}{u^{2}}\tilde{g}_{ij}.
\end{equation*}The Gauss and Codazzi equations are given by
\begin{align*}
R_{ijkl}=-(\delta_{ik}\delta_{jl}-\delta_{il}\delta_{jk})+h_{ik}h_{jl}-h_{il}h_{jk}, 
\end{align*}
\begin{align*}
\nabla_{k}h_{ij}=\nabla_{j}h_{ik}.
\end{align*}The Ricci identity is given by \begin{align}\label{Ricci}
\nabla_{ij}h_{kl}=&\ \nabla_{kl}h_{ij}-h_{ml}(h_{im}h_{kj}-h_{ij}h_{mk})-h_{mj}(h_{mi}h_{kl}-h_{il}h_{mk})\\ \nonumber
&-h_{ml}(\delta_{ij}\delta_{km}-\delta_{ik}\delta_{jm}) -h_{mj}(\delta_{il}\delta_{km}-\delta_{ik}\delta_{lm}).
\end{align}\par The following lemma is repeatedly employed in the proof of the main theorem.
\begin{lemm}In a local orthonormal frame $e_{1},\cdots,e_{n}$, we have that\label{L1}
\begin{align*}
    C\leq \nu^{n+1}\leq 1, \quad where\;\nu^{n+1}=\frac{1}{\sqrt{1+|Du|^{2}}}
\end{align*}
\begin{align*}
\sum_i \frac{u_i^2}{u^2}=1-(\nu^{n+1})^2\leq 1,\quad  \nabla_{i}\nu^{n+1}=\frac{u_i}{u}(\nu^{n+1}-\kappa_i),
\end{align*}
\begin{align*}
\sum_i S_{n}^{ii}\nabla_{ii}\nu^{n+1} =&\ 2\sum_i S_{n}^{ii}\frac{u_i}{u} \nabla_{i}\nu^{n+1}+(n\sigma-\sigma_{n-1}) (1+(\nu^{n+1})^2)\\
&-\nu^{n+1}\left(\sum_i S_{n}^{ii}+\sum_i S_{n}^{ii}\kappa_i^2\right).
\end{align*}
\end{lemm}
\begin{proof}
   The first inequality can be found in proposition 4.1 of \cite{GS10}. The second equality follows immediately from definition of $\nu^{n+1}$, while the third one is given by equation (3.9) in \cite{BJL}. For the fouth equality, by combining (3.9) in \cite{BJL} with Lemma \ref{lem2.1} (v), we obtain\begin{align*}
\sum_i S_{n}^{ii}\nabla_{ii}\nu^{n+1} &=\ 2\sum_i S_{n}^{ii}\frac{u_i}{u} \nabla_{i}\nu^{n+1}+(1+(\nu^{n+1})^2)\sum_{i}S_{n}^{ii}\kappa_{i}\\
&\;\;\;-\nu^{n+1}\left(\sum_i S_{n}^{ii}+\sum_i S_{n}^{ii}\kappa_i^2\right)\\
&=\ 2\sum_i S_{n}^{ii}\frac{u_i}{u} \nabla_{i}\nu^{n+1}+(n\sigma-\sigma_{n-1}) (1+(\nu^{n+1})^2)\\
&\;\;\;-\nu^{n+1}\left(\sum_i S_{n}^{ii}+\sum_i S_{n}^{ii}\kappa_i^2\right).
\end{align*}
\end{proof}
\section{Proof of main theorems}
\par In this section, we present the proof of Theorem \ref{theo1}, from which we obtain Theorem \ref{theo2}.

\textit{Proof of Theorem \ref{theo1}.}
\begin{proof}

    We consider the test function proposed by Lu \cite{S. Lu} 
    \begin{align*}
Q=\ln \kappa_1-N\ln \nu^{n+1},
\end{align*}
where $\kappa_1$ is the largest principle curvature and $N$ is a large constant to be determined later.\par Suppose that $Q$ attains its maximum at an interior point $X_{0}$ of $\Sigma$. Should the principal curvature $\kappa_1$ possess multiplicity greater than $1$, the function $Q$ fails to remain smooth at $X_0$. Consequently, we employ the curvature perturbation technique(cf. \cite{J.Chu}) to resolve this issue. We select a local orthonormal frame ${e_1,\dots,e_n}$ near $X_0$ satisfying the following conditions at $X_0$:
\begin{align*}
\nabla_{e_i}e_j = 0, \quad h_{ij} = \delta_{ij}\kappa_i, \quad \text{with} \quad \kappa_1 \geq \cdots \geq \kappa_n
\end{align*}
In a neighborhood of $X_0$, we introduce a new tensor $B$ defined for tangent vectors $V_{1},V_{2}$ by
\begin{align*}
B(V_1,V_2)=g(V_1,V_2)-g(V_1,e_1)g(V_2,e_1),
\end{align*}Denote $B_{ij}=B(e_{i},e_{j})$, we construct the modified second fundamental form
\begin{align*}
\tilde{h}_{ij}=h_{ij}-B_{ij}.
\end{align*}Let $\tilde{\kappa}_{1}\geq \cdots \geq \tilde{\kappa}_{n}$ denote the eigenvalues of $\tilde{h}_{ij}$. Near $X_{0}$, we have $\kappa_1\geq \tilde{\kappa}_1$ and at $X_{0}$, the following holds: \begin{align*}
\tilde{\kappa}_i=\begin{cases}
\kappa_1,\quad &i=1,\\
\kappa_i-1, & i>1.
\end{cases}
\end{align*}

We now replace $\kappa_{1}$ with $\tilde{\kappa}_{1}$(smooth at $X_{0}$) in function $Q$ and consider the new test function
\begin{align*}
\tilde{Q}=\ln\tilde{\kappa}_1-N\ln \nu^{n+1}.
\end{align*}Note that $\tilde{Q}$ also attains its maximum at $X_{0}$. Therefore, by the maximum principle, at $X_{0}$, we have
\begin{align}\label{(3.1)}
0=\nabla_{i}\tilde{Q}=\frac{\nabla_{i}\tilde{\kappa}_1}{\tilde{\kappa}_1}-N\frac{\nabla_{i}\nu^{n+1}}{\nu^{n+1}},
\end{align}
\begin{align}\label{(3.2)}
0\geq \nabla_{ii}\tilde{Q}=\frac{\nabla_{ii}\tilde{\kappa}_1}{\tilde{\kappa}_1}-\frac{(\nabla_{i}\tilde{\kappa}_{1})^{2}}{\tilde{\kappa}_1^2}-N\frac{\nabla_{ii}\nu^{n+1}}{\nu^{n+1}}+N\frac{(\nabla_{i}\nu^{n+1})^2}{(\nu^{n+1})^2}.
\end{align}\par We first address the fourth-order terms. By the definition of tensor $B$ and $\nabla_{e_{i}}e_{j}=0$, we have 
\begin{equation*}
    \nabla_{p}B_{ij}=0,\quad \nabla_{ii}B_{11}=0\;\;\text{at}\;X_{0}.
\end{equation*} It follows that\begin{equation*}
    \nabla_{p}\tilde{h}_{ij}=\nabla_{p}h_{ij},\quad\nabla_{ii}\tilde{h}_{11}=\nabla_{ii}h_{11}.
\end{equation*}
Direct calculation shows that \begin{equation*}
    \begin{split}
        &\nabla_{i}\tilde{\kappa}_{1}=\frac{\partial \tilde{\kappa}_{1}}{\partial \tilde{h}_{pq}}\nabla_{i}\tilde{h}_{pq}=\nabla_{i}\tilde{h}_{11}=\nabla_{i}h_{11}\\
        &\nabla_{ii}\tilde{\kappa}_{1}=\frac{\partial \tilde{\kappa}_{1}}{\partial \tilde{h}_{pq}}\nabla_{ii}\tilde{h}_{pq}+\frac{\partial ^{2} \tilde{\kappa}_{1}}{\partial \tilde{h}_{pq} \partial \tilde{h}_{rs}}\nabla_{i}\tilde{h}_{pq}\nabla_{i}\tilde{h}_{rs}\\
        &\;\;\;\;\;\;\;\;=\nabla_{ii}\tilde{h}_{11}+2\sum_{p>1}\frac{(\nabla_{i}\tilde{h}_{1p})^{2}}{\kappa_{1}-\tilde{\kappa}_{p}}=\nabla_{ii}h_{11}+2\sum_{p>1}\frac{(\nabla_{i}h_{1p})^{2}}{\kappa_{1}-\tilde{\kappa}_{p}},
    \end{split}
\end{equation*}where we heve used the fact $\tilde{\kappa}_{1}=\kappa_{1}$ at $X_{0}$.
Substituting the last equation above into \eqref{(3.2)}\begin{equation}\label{(3.3)}
    0\geq \frac{\nabla_{ii}h_{11}}{\kappa_{1}}+2\sum_{p>1}\frac{(\nabla_{i}h_{1p})^{2}}{\kappa_{1}(\kappa_{1}-\tilde{\kappa}_{p})}-\frac{(\nabla_{i}h_{11})^{2}}{\kappa_1^2}-N\frac{\nabla_{ii}\nu^{n+1}}{\nu^{n+1}}.
\end{equation}By Ricci identity \eqref{Ricci}, we have
\begin{align*}
\nabla_{ii}h_{11}=\nabla_{11}h_{ii}+\kappa_1 ^2\kappa_i -\kappa_1\kappa_i ^2-\kappa_1+\kappa_i.
\end{align*}Combining with \eqref{(3.3)} yields
\begin{equation}\label{(3.4)}
    \begin{split}
        0&\geq \frac{\nabla_{11}h_{ii}}{\kappa_{1}}+2\sum_{p>1}\frac{(\nabla_{i}h_{1p})^{2}}{\kappa_{1}(\kappa_{1}-\tilde{\kappa}_{p})}-\frac{(\nabla_{i}h_{11})^{2}}{\kappa_1^2}\\
        &\;\;\;-N\frac{\nabla_{ii}\nu^{n+1}}{\nu^{n+1}}+\kappa_{1}\kappa_{i}-\kappa_{i}^{2}-1+\frac{\kappa_{i}}{\kappa_{1}}.
    \end{split}
\end{equation}
Contracting \eqref{(3.4)} with $S_{n}^{ii}$, we have
\begin{equation}\label{(3.5)}
    \begin{split}
        0&\geq \sum_{i}\frac{S_{n}^{ii}\nabla_{11}h_{ii}}{\kappa_{1}}+2\sum_{i}\sum_{p>1}\frac{S_{n}^{ii}(\nabla_{i}h_{1p})^{2}}{\kappa_{1}(\kappa_{1}-\tilde{\kappa}_{p})}-\sum_{i}\frac{S_{n}^{ii}(\nabla_{i}h_{11})^{2}}{\kappa_1^2}\\
        &\;\;\;-N\sum_{i}\frac{S_{n}^{ii}\nabla_{ii}\nu^{n+1}}{\nu^{n+1}}+\kappa_{1}\sum_{i}S_{n}^{ii}\kappa_{i}-\sum_{i}S_{n}^{ii}\kappa_{i}^{2}-\sum_{i}S_{n}^{ii}+\sum_{i}\frac{S_{n}^{ii}\kappa_{i}}{\kappa_{1}}.
    \end{split}
\end{equation}
Differentiating \eqref{S_n} twice, we have\begin{equation}\label{(3.6)}
    \sum_{i}S_{n}^{ii}\nabla_{11}h_{ii}=-\sum_{p,q,r,s}S_{n}^{pq,rs}\nabla_{1}h_{pq}\nabla_{1}h_{rs}.
\end{equation}
Substituting \eqref{(3.6)} into \eqref{(3.5)} yields
\begin{equation}\label{(3.7)}
    \begin{split}
        0&\geq -\frac{1}{\kappa_{1}}\sum_{p,q,r,s}S_{n}^{pq,rs}\nabla_{1}h_{pq}\nabla_{1}h_{rs}+2\sum_{i}\sum_{p>1}\frac{S_{n}^{ii}(\nabla_{i}h_{1p})^{2}}{\kappa_{1}(\kappa_{1}-\tilde{\kappa}_{p})}-\sum_{i}\frac{S_{n}^{ii}(\nabla_{i}h_{11})^{2}}{\kappa_1^2}\\
        &\;\;\;-N\sum_{i}\frac{S_{n}^{ii}\nabla_{ii}\nu^{n+1}}{\nu^{n+1}}+\left(\kappa_{1}+\frac{1}{\kappa_{1}}\right)\sum_{i}S_{n}^{ii}\kappa_{i}-\sum_{i}S_{n}^{ii}\kappa_{i}^{2}-\sum_{i}S_{n}^{ii}.
    \end{split}
\end{equation}
By Lemma \ref{lem2.1} (v) and Lemma \ref{L1}, we have
\begin{equation}\label{(3.8)}
    \begin{split}
        0&\geq -\frac{1}{\kappa_{1}}\sum_{p,q,r,s}S_{n}^{pq,rs}\nabla_{1}h_{pq}\nabla_{1}h_{rs}+2\sum_{i}\sum_{p>1}\frac{S_{n}^{ii}(\nabla_{i}h_{1p})^{2}}{\kappa_{1}(\kappa_{1}-\tilde{\kappa}_{p})}-\sum_{i}\frac{S_{n}^{ii}(\nabla_{i}h_{11})^{2}}{\kappa_1^2}\\
        &\;\;\;-2N\sum_{i}S_{n}^{ii}\frac{u_{i}}{u}\frac{\nabla_{i}\nu^{n+1}}{\nu^{n+1}}+\left(\kappa_{1}+\frac{1}{\kappa_{1}}\right)\sum_{i}S_{n}^{ii}\kappa_{i}-n\sigma\frac{1+(\nu^{n+1})^{2}}{\nu^{n+1}}\\
        &\;\;\;+(N-1)\left(\sum_{i}S_{n}^{ii}+\sum_{i}S_{n}^{ii}\kappa_{i}^{2}\right).
    \end{split}
\end{equation}
By the first inequality in Lemma \ref{L1}, we have
\begin{equation}\label{(3.9)}
    \begin{split}
        0&\geq -\frac{1}{\kappa_{1}}\sum_{p,q,r,s}S_{n}^{pq,rs}\nabla_{1}h_{pq}\nabla_{1}h_{rs}+2\sum_{i}\sum_{p>1}\frac{S_{n}^{ii}(\nabla_{i}h_{1p})^{2}}{\kappa_{1}(\kappa_{1}-\tilde{\kappa}_{p})}-\sum_{i}\frac{S_{n}^{ii}(\nabla_{i}h_{11})^{2}}{\kappa_1^2}\\
        &\;\;\;-2N\sum_{i}S_{n}^{ii}\frac{u_{i}}{u}\frac{\nabla_{i}\nu^{n+1}}{\nu^{n+1}}+\left(\kappa_{1}+\frac{1}{\kappa_{1}}\right)\sum_{i}S_{n}^{ii}\kappa_{i}-CN\\
        &\;\;\;+(N-1)\left(\sum_{i}S_{n}^{ii}+\sum_{i}S_{n}^{ii}\kappa_{i}^{2}\right),
    \end{split}
\end{equation}where $C=C(n,\Omega,\sigma,A,\alpha)>0$ denotes a universal constant that may vary between lines. \par By Lemma 2.5 in \cite{PR}, we have\begin{equation*}
\sum_{p,q,r,s}S_{n}^{pq,rs}\nabla_{1}h_{pq}\nabla_{1}h_{rs}=\sum_{p\neq q}S_{n}^{pp,qq}\nabla_{1}h_{pp}\nabla_{1}h_{qq}-\sum_{p\neq q}S_{n}^{pp,qq}(\nabla_{1}h_{pq})^{2}
\end{equation*}Inserting into \eqref{(3.9)} yields\begin{equation}\label{(3.10)}
    \begin{split}
        0&\geq -\sum_{p\neq q}\frac{S_{n}^{pp,qq}\nabla_{1}h_{pp}\nabla_{1}h_{qq}}{\kappa_{1}}+\sum_{p\neq q}\frac{S_{n}^{pp,qq}(\nabla_{1}h_{pq})^{2}}{\kappa_{1}}+2\sum_{i}\sum_{p>1}\frac{S_{n}^{ii}(\nabla_{i}h_{1p})^{2}}{\kappa_{1}(\kappa_{1}-\tilde{\kappa}_{p})}\\
        &\;\;\;-2N\sum_{i}S_{n}^{ii}\frac{u_{i}}{u}\frac{\nabla_{i}\nu^{n+1}}{\nu^{n+1}}+\left(\kappa_{1}+\frac{1}{\kappa_{1}}\right)\sum_{i}S_{n}^{ii}\kappa_{i}-\sum_{i}\frac{S_{n}^{ii}(\nabla_{i}h_{11})^{2}}{\kappa_1^2}\\
        &\;\;\;+(N-1)\left(\sum_{i}S_{n}^{ii}+\sum_{i}S_{n}^{ii}\kappa_{i}^{2}\right)-CN.
    \end{split}
\end{equation}Assume $\kappa_1$ has multiplicity $m$, then we have\begin{equation*}
    \sum_{p\neq q}\frac{S_{n}^{pp,qq}(\nabla_{1}h_{pq})^{2}}{\kappa_{1}}\geq 2\sum_{p>m}\frac{S_{n}^{11,pp}(\nabla_{p}h_{11})^{2}}{\kappa_1 }\geq 2\sum_{p>m}\frac{(S_{n}^{pp}-S_{n}^{11})(\nabla_{p}h_{11})^{2}}{\kappa_1 (\kappa_1 -\tilde{\kappa}_p )} 
\end{equation*}Substituting into \eqref{(3.10)} gives\begin{equation}\label{(3.11)}
    \begin{split}
        0&\geq -\sum_{p\neq q}\frac{S_{n}^{pp,qq}\nabla_{1}h_{pp}\nabla_{1}h_{qq}}{\kappa_{1}}+2\sum_{p>m}\frac{(S_{n}^{pp}-S_{n}^{11})(\nabla_{p}h_{11})^{2}}{\kappa_1 (\kappa_1 -\tilde{\kappa}_p )}\\
        &\;\;\;-2N\sum_{i}S_{n}^{ii}\frac{u_{i}}{u}\frac{\nabla_{i}\nu^{n+1}}{\nu^{n+1}}+\left(\kappa_{1}+\frac{1}{\kappa_{1}}\right)\sum_{i}S_{n}^{ii}\kappa_{i}+2\sum_{i}\sum_{p>1}\frac{S_{n}^{ii}(\nabla_{i}h_{1p})^{2}}{\kappa_{1}(\kappa_{1}-\tilde{\kappa}_{p})}\\
        &\;\;\;-\sum_{i}\frac{S_{n}^{ii}(\nabla_{i}h_{11})^{2}}{\kappa_1^2}+(N-1)\left(\sum_{i}S_{n}^{ii}+\sum_{i}S_{n}^{ii}\kappa_{i}^{2}\right)-CN.
    \end{split}
\end{equation}
Discarding certain nonnegative terms, we obtain 
\begin{align*}
2\sum_i\sum_{p\neq 1}\frac{S_{n}^{ii}(\nabla_{i}h_{1p})^2}{\kappa_1 (\kappa_1-\tilde{\kappa}_p)}&\geq 2\sum_{p\neq 1}\frac{S_{n}^{pp}(\nabla_{p}h_{1p}^2)}{\kappa_1 (\kappa_1-\tilde{\kappa}_p)}+2\sum_{p\neq 1}\frac{S_{n}^{11}(\nabla_{1}h_{1p})^2}{\kappa_1 (\kappa_1-\tilde{\kappa}_p)}\\
&=2\sum_{i\neq 1}\frac{S_{n}^{ii}(\nabla_{1}h_{ii})^2}{\kappa_1 (\kappa_1-\tilde{\kappa}_i)}+2\sum_{i\neq 1}\frac{S_{n}^{11}(\nabla_{i}h_{11})^2}{\kappa_1 (\kappa_1-\tilde{\kappa}_i)}.
\end{align*}
Combining this with \eqref{(3.11)}  yields
     \begin{equation}\label{(3.12)}
    \begin{split}
        0&\geq -\sum_{p\neq q}\frac{S_{n}^{pp,qq}\nabla_{1}h_{pp}\nabla_{1}h_{qq}}{\kappa_{1}}+2\sum_{i>m}\frac{(S_{n}^{ii}-S_{n}^{11})(\nabla_{i}h_{11})^{2}}{\kappa_1 (\kappa_1 -\tilde{\kappa}_i )}+2\sum_{i\neq 1}\frac{S_{n}^{11}(\nabla_{i}h_{11})^2}{\kappa_1 (\kappa_1-\tilde{\kappa}_i)}\\
        &\;\;\;-2N\sum_{i}S_{n}^{ii}\frac{u_{i}}{u}\frac{\nabla_{i}\nu^{n+1}}{\nu^{n+1}}+\left(\kappa_{1}+\frac{1}{\kappa_{1}}\right)\sum_{i}S_{n}^{ii}\kappa_{i}+2\sum_{i\neq 1}\frac{S_{n}^{ii}(\nabla_{1}h_{ii})^2}{\kappa_1 (\kappa_1-\tilde{\kappa}_i)}\\
        &\;\;\;-\sum_{i}\frac{S_{n}^{ii}(\nabla_{i}h_{11})^{2}}{\kappa_1^2}+(N-1)\left(\sum_{i}S_{n}^{ii}+\sum_{i}S_{n}^{ii}\kappa_{i}^{2}\right)-CN.
    \end{split}
\end{equation}
By Lemma \ref{lem2.1} (i) and $\kappa\in \widetilde{\Gamma}_{n}$, we have\begin{equation*}
    S_{n}^{11,22,\cdots,n-1n-1}(\kappa)=S_{1}(\kappa|12\cdots n-1)=1+\alpha\kappa_{n}>0.
\end{equation*}We thus conclude that $\kappa_{i}>-\frac{1}{\alpha},\;i=1,2,\cdots,n$.
\par We now address the troublesome third-order terms. We claim that\begin{equation}\label{(3.13)}
    \begin{split}
        &2\sum_{i>m}\frac{(S_{n}^{ii}-S_{n}^{11})(\nabla_{i}h_{11})^{2}}{\kappa_1 (\kappa_1 -\tilde{\kappa}_i )}+2\sum_{i\neq 1}\frac{S_{n}^{11}(\nabla_{i}h_{11})^2}{\kappa_1 (\kappa_1-\tilde{\kappa}_i)}-\sum_{i\neq 1}\frac{S_{n}^{ii}(\nabla_{i}h_{11})^{2}}{\kappa_{1}^{2}}\\
        & \geq \frac{1}{2}\sum_{i\neq 1}\frac{S_{n}^{ii}(\nabla_{i}h_{11})^{2}}{\kappa_{1}^{2}}.
    \end{split}
\end{equation}Indeed, assume that $\kappa_{1}>\frac{3}{\alpha}+3$, we have \begin{equation*}
   \begin{split}
        &\;\;\;\;2\sum_{i>m}\frac{(S_{n}^{ii}-S_{n}^{11})(\nabla_{i}h_{11})^{2}}{\kappa_1 (\kappa_1 -\tilde{\kappa}_i )}+2\sum_{i\neq 1}\frac{S_{n}^{11}(\nabla_{i}h_{11})^2}{\kappa_1 (\kappa_1-\tilde{\kappa}_i)}-\sum_{i\neq 1}\frac{S_{n}^{ii}(\nabla_{i}h_{11})^{2}}{\kappa_{1}^{2}}\\
        &=2\sum_{i>m}\frac{(S_{n}^{ii}-S_{n}^{11})(\nabla_{i}h_{11})^{2}}{\kappa_1 (\kappa_1 -\tilde{\kappa}_i )}+2\sum_{1<i\leq m}\frac{S_{n}^{11}(\nabla_{i}h_{11})^{2}}{\kappa_{1}(\kappa_{1}-\tilde{\kappa}_{1})}+2\sum_{i>m}\frac{S_{n}^{11}(\nabla_{i}h_{11})^{2}}{\kappa_{1}(\kappa_{1}-\tilde{\kappa}_{1})}\\
        &\;\;\;\;-\sum_{1<i\leq m}\frac{S_{n}^{ii}(\nabla_{i}h_{11})^{2}}{\kappa_{1}^{2}}-\sum_{i>m}\frac{S_{n}^{ii}(\nabla_{i}h_{11})^{2}}{\kappa_{1}^{2}}\\
        &=\sum_{i>m}\frac{S_{n}^{11}(\nabla_{i}h_{11})^{2}}{\kappa_{1}}\left(\frac{2}{\kappa_{1}-\tilde{\kappa}_{i}}-\frac{2}{\kappa_{1}-\tilde{\kappa}_{i}}\right)+\sum_{i>m}\frac{S_{n}^{ii}(\nabla_{i}h_{11})^{2}}{\kappa_{1}}\left(\frac{2}{\kappa_{1}-\tilde{\kappa}_{i}}-\frac{1}{\kappa_{1}}\right)\\
        &\;\;\;\;+2\sum_{1<i\leq m}\frac{S_{n}^{11}(\nabla_{i}h_{11})^{2}}{\kappa_{1}(\kappa_{1}-\tilde{\kappa}_{1})}-\sum_{1<i\leq m}\frac{S_{n}^{ii}(\nabla_{i}h_{11})^{2}}{\kappa_{1}^{2}}\\
        &=\sum_{i>m}\frac{S_{n}^{ii}(\nabla_{i}h_{11})^{2}}{\kappa_{1}}\frac{\kappa_{1}+\tilde{\kappa}_{i}}{\kappa_{1}(\kappa_{1}-\tilde{\kappa}_{i})}+\sum_{1<i\leq m}\frac{S_{n}^{11}(\nabla_{i}h_{11})^{2}}{\kappa_{1}}\left(\frac{2}{\kappa_{1}-\tilde{\kappa}_{i}}-\frac{1}{\kappa_{1}}\right)\\
        &=\left(2\kappa_{1}-1\right) \sum_{1<i\leq m}\frac{S_{n}^{ii}(\nabla_{i}h_{11})^{2}}{\kappa_{1}^{2}}+\sum_{i>m}\frac{S_{n}^{ii}(\nabla_{i}h_{11})^{2}}{\kappa_{1}^{2}}\frac{\kappa_{1}+\tilde{\kappa}_{i}}{\kappa_{1}-\tilde{\kappa}_{i}}\\
        & \geq \frac{1}{2}\sum_{1<i\leq m}\frac{S_{n}^{ii}(\nabla_{i}h_{11})^{2}}{\kappa_{1}^{2}}+\sum_{i> m}\frac{S_{n}^{ii}(\nabla_{i}h_{11})^{2}}{\kappa_{1}^{2}} \frac{a+b-2}{a-b},
   \end{split} 
\end{equation*}where $a=\kappa_{1}+1,\;b=\kappa_{i}$. Direct calculation shows 
\begin{equation*}
    \begin{split}
        \sum_{i>m}\frac{S_{n}^{ii}(\nabla_{i}h_{11})^{2}}{\kappa_{1}^{2}} \frac{a+b-2}{a-b}&=\sum_{i>m}\frac{S_{n}^{ii}(\nabla_{i}h_{11})^{2}}{\kappa_{1}^{2}} \left(\frac{a+b-2}{a-b}-\frac{1}{2}\right)+\frac{1}{2}\sum_{i>m}\frac{S_{n}^{ii}(\nabla_{i}h_{11})^{2}}{\kappa_{1}^{2}}\\
        &=\sum_{i>m}\frac{S_{n}^{ii}(\nabla_{i}h_{11})^{2}}{\kappa_{1}^{2}}\frac{a+3b-4}{2(a-b)}+\frac{1}{2}\sum_{i>m}\frac{S_{n}^{ii}(\nabla_{i}h_{11})^{2}}{\kappa_{1}^{2}}\\
        &>\frac{1}{2}\sum_{i>m}\frac{S_{n}^{ii}(\nabla_{i}h_{11})^{2}}{\kappa_{1}^{2}},
    \end{split}
\end{equation*}where we used the fact $a>\frac{3}{\alpha}+4$ and $a-b>0$. \par Thus, claim \eqref{(3.13)} is valid. Inserting \eqref{(3.13)} into \eqref{(3.12)} yields 
\begin{equation}\label{(3.14)}
    \begin{split}
        0&\geq -\sum_{p\neq q}\frac{S_{n}^{pp,qq}\nabla_{1}h_{pp}\nabla_{1}h_{qq}}{\kappa_{1}}+2\sum_{i\neq 1}\frac{S_{n}^{ii}(\nabla_{1}h_{ii})^2}{\kappa_1 (\kappa_1-\tilde{\kappa}_i)}-\frac{S_{n}^{11}(\nabla_{1}h_{11})^{2}}{\kappa_{1}^{2}}\\
        &\;\;\;+\frac{1}{2}\sum_{i\neq 1}\frac{S_{n}^{ii}(\nabla_{i}h_{11})^{2}}{\kappa_{1}^{2}}-2N\sum_{i}S_{n}^{ii}\frac{u_{i}}{u}\frac{\nabla_{i}\nu^{n+1}}{\nu^{n+1}}+\left(\kappa_{1}+\frac{1}{\kappa_{1}}\right)\sum_{i}S_{n}^{ii}\kappa_{i}\\
        &\;\;\;+(N-1)\left(\sum_{i}S_{n}^{ii}+\sum_{i}S_{n}^{ii}\kappa_{i}^{2}\right)-CN.
    \end{split}
\end{equation}Since $\kappa_{i}>-\frac{1}{\alpha},\;i=1,2,\cdots,n$, we have \begin{equation*}
    \frac{2}{\kappa_{1}(\kappa_{1}-\tilde{\kappa}_{1})} \geq \frac{2}{\kappa_{1}(\kappa_{1}+\frac{1}{\alpha}+1)}\geq \frac{2\alpha}{\alpha+1}\frac{1}{\kappa_{1}^{2}}.
\end{equation*}Combining \eqref{(3.14)}, we obtain
\begin{equation}\label{(3.15)}
    \begin{split}
        0&\geq -\sum_{p\neq q}\frac{S_{n}^{pp,qq}\nabla_{1}h_{pp}\nabla_{1}h_{qq}}{\kappa_{1}}+C\sum_{i\neq 1}\frac{S_{n}^{ii}(\nabla_{1}h_{ii})^2}{\kappa_1^{2}}-\frac{S_{n}^{11}(\nabla_{1}h_{11})^{2}}{\kappa_{1}^{2}}\\
        &\;\;\;+\frac{1}{2}\sum_{i\neq 1}\frac{S_{n}^{ii}(\nabla_{i}h_{11})^{2}}{\kappa_{1}^{2}}-2N\sum_{i}S_{n}^{ii}\frac{u_{i}}{u}\frac{\nabla_{i}\nu^{n+1}}{\nu^{n+1}}+\left(\kappa_{1}+\frac{1}{\kappa_{1}}\right)\sum_{i}S_{n}^{ii}\kappa_{i}\\
        &\;\;\;+(N-1)\left(\sum_{i}S_{n}^{ii}+\sum_{i}S_{n}^{ii}\kappa_{i}^{2}\right)-CN.
    \end{split}
\end{equation}Applying Lemma \ref{lem2.3}, we have
\begin{equation}\label{(3.16)}
    \begin{split}
        0&\geq \frac{1}{2}\sum_{i\neq 1}\frac{S_{n}^{ii}(\nabla_{i}h_{11})^{2}}{\kappa_{1}^{2}}-2N\sum_{i}S_{n}^{ii}\frac{u_{i}}{u}\frac{\nabla_{i}\nu^{n+1}}{\nu^{n+1}}+\left(\kappa_{1}+\frac{1}{\kappa_{1}}\right)\sum_{i}S_{n}^{ii}\kappa_{i}\\
        &\;\;\;+(N-1)\left(\sum_{i}S_{n}^{ii}+\sum_{i}S_{n}^{ii}\kappa_{i}^{2}\right)-CN.
    \end{split}
\end{equation}By Lemma \ref{lem2.1} and $\sigma_{n}>-A$, we have\begin{equation*}
    \begin{split}
        \left(\kappa_{1}+\frac{1}{\kappa_{1}}\right)\sum_{i}S_{n}^{ii}\kappa_{i}&=\left(\kappa_{1}+\frac{1}{\kappa_{1}}\right)\left[(n-1)\sigma_{n-1}+n\alpha\sigma_{n}\right] \\
        &\geq -n\alpha\left(\kappa_{1}+\frac{1}{\kappa_{1}}\right)A\geq -C\kappa_{1}.
    \end{split}
\end{equation*} Inserting into \eqref{(3.16)} yields 
\begin{equation}\label{(3.17)}
    \begin{split}
        0&\geq \frac{1}{2}\sum_{i\neq 1}\frac{S_{n}^{ii}(\nabla_{i}h_{11})^{2}}{\kappa_{1}^{2}}-2N\sum_{i}S_{n}^{ii}\frac{u_{i}}{u}\frac{\nabla_{i}\nu^{n+1}}{\nu^{n+1}}-C\kappa_{1}\\
        &\;\;\;+(N-1)\left(\sum_{i}S_{n}^{ii}+\sum_{i}S_{n}^{ii}\kappa_{i}^{2}\right)-CN.
    \end{split}
\end{equation}By Lemma \ref{L1} and $\kappa_{1}$ sufficiently large, we can discard a positive term\begin{equation*}
    -2NS_{n}^{11}\frac{u_{1}}{u}\frac{\nabla_{1}\nu^{n+1}}{\nu^{n+1}}=-2NS_{n}^{11}\frac{u_{1}^{2}}{u^{2}}\frac{(\nu^{n+1}-\kappa_{1})}{\nu^{n+1}}
\end{equation*}Consequently, \eqref{(3.17)} implies
    \begin{equation}\label{(3.18)}
    \begin{split}
        0&\geq \frac{1}{2}\sum_{i\neq 1}\frac{S_{n}^{ii}(\nabla_{i}h_{11})^{2}}{\kappa_{1}^{2}}-2N\sum_{i\neq 1}S_{n}^{ii}\frac{u_{i}}{u}\frac{\nabla_{i}\nu^{n+1}}{\nu^{n+1}}-C\kappa_{1}\\
        &\;\;\;+(N-1)\left(\sum_{i}S_{n}^{ii}+\sum_{i}S_{n}^{ii}\kappa_{i}^{2}\right)-CN.
    \end{split}
\end{equation}By critical equation \eqref{(3.1)}, we have 
\begin{equation}\label{(3.19)}
    \begin{split}
        0&\geq \frac{N^{2}}{2}\sum_{i\neq 1}\frac{S_{n}^{ii}(\nabla_{i}\nu^{n+1})^{2}}{(\nu^{n+1})^{2}}-2N\sum_{i\neq 1}S_{n}^{ii}\frac{u_{i}}{u}\frac{\nabla_{i}\nu^{n+1}}{\nu^{n+1}}-C\kappa_{1}\\
        &\;\;\;+(N-1)\left(\sum_{i}S_{n}^{ii}+\sum_{i}S_{n}^{ii}\kappa_{i}^{2}\right)-CN\\
        &\geq -2\sum_{i\neq 1}S_{n}^{ii}\frac{u_{i}^{2}}{u^{2}}+(N-1)\left(\sum_{i}S_{n}^{ii}+\sum_{i}S_{n}^{ii}\kappa_{i}^{2}\right)-C\kappa_{1}-CN.
    \end{split}
\end{equation}Choose $N$ sufficiently large, we have 
\begin{equation}\label{(3.20)}
    \begin{split}
        0 \geq (N-1)\sum_{i}S_{n}^{ii}\kappa_{i}^{2}-C\kappa_{1}-CN.
    \end{split}
\end{equation}By Lemma \ref{lem2.2}, we have 
\begin{equation}\label{(3.21)}
    \begin{split}
         0 &\geq (N-1)\sum_{i}S_{n}^{ii}\kappa_{i}^{2}-C\kappa_{1}-CN\\
         &\geq (N-1)S_{n}^{11}\kappa_{1}^{2}-C\kappa_{1}-CN\\
         &\geq (N-1)\theta\sigma\kappa_{1}-C\kappa_{1}-CN
    \end{split}
\end{equation}For sufficiently large $N$, \eqref{(3.21)} implies $\kappa_{1} \leq C$.
\end{proof}

\end{document}